\documentclass[a4paper,11pt]{article}

\usepackage[a4paper,margin=1in]{geometry}
\usepackage{amsmath, amssymb, amsthm, amsfonts}
\usepackage[T1]{fontenc}
\usepackage{lmodern}
\usepackage{graphicx}
\usepackage{enumitem}
\usepackage{authblk}
\usepackage{float}

\usepackage{hyperref}
\hypersetup{
	colorlinks=true,
	linkcolor=blue,
	citecolor=blue,
	urlcolor=blue
}
\usepackage{setspace}
\usepackage{tikz}
\usepackage[mathscr]{euscript}
\newtheorem{theorem}{Theorem}[section]
\newtheorem{lemma}[theorem]{Lemma}
\newtheorem{corollary}[theorem]{Corollary}
\newtheorem{proposition}[theorem]{Proposition}
\numberwithin{equation}{section}
\theoremstyle{definition}

\newtheorem{observation}[theorem]{Observation}

\newcommand{\ig}{\mathcal{I}(G)}
\newcommand{\V}{\mathscr{M}}

\newcommand{\vu}{\mathscr{M}_{u}}
\newcommand{\g}{\mathcal{G}}

\newcommand{\ii}{\mathcal{I}}

\newcommand{\igu}{\mathcal{I}(G - N_{G}[u])}
\newcommand{\igv}{\mathcal{I}(G - N_{G}[v])}
\newcommand{\igij}{\mathcal{I}(G_{i,j} - N_{G_{i,j}}[u])}

\title{A recursive construction of an acyclic matching on the independence complex of a graph with a simplicial vertex}

\author{Sucharita Barik\thanks{\texttt{sucharita.barik.126@tcgcrest.org}}}

\author{Anupam Mondal\thanks{\texttt{anupam.mondal@tcgcrest.org}}}

\author{Sajal Mukherjee\thanks{\texttt{sajal.mukherjee@tcgcrest.org}}}

\affil{\small Institute for Advancing Intelligence (IAI), TCG CREST, Kolkata--700091, India}
\affil{\small Academy of Scientific and Innovative Research (AcSIR), Ghaziabad--201002, India}
\date{}

\begin{document}
	
	\maketitle
	
	\date{}
	
	
	\begin{abstract}
		We provide a recursive construction of an acyclic matching (also known as a gradient vector field, an equivalent notion to a discrete Morse function) on the independence complex of a graph with a simplicial vertex using given acyclic matchings on the independence complexes of specific subgraphs. As an application, we determine the homotopy type of the independence complexes of the family of chordal graphs and of a class of graphs generalising the comparability graphs of grid posets in an algorithmic and combinatorial manner via discrete Morse theory, some of which were previously obtained by sophisticated homotopy theoretic techniques. Even when the homotopy type is not easily determinable, our construction may be applied to obtain a pre-processing framework for efficient homology computation.
	\end{abstract}
	
	

	\section{Introduction}
	Discrete Morse theory, introduced by Forman \cite{formanmain, formanuser}, is a powerful combinatorial adaptation of classical Morse theory, providing an effective and algorithmically tractable framework for analysing the topology (such as homotopy, homology, etc.) of an (abstract) simplicial complex. Since its introduction, this theory has proven to be a powerful tool with applications across diverse areas of both theoretical and applied mathematics. 
	
	Although discrete Morse theory provides computationally convenient techniques in contrast to the previously known topological machinery, it usually requires an efficient discrete Morse function to begin with. Rather, in practice, instead of discrete Morse functions, it is more convenient to work with an equivalent notion of \emph{acyclic matchings}, also known as \emph{gradient vector fields} on a given complex. The strength of discrete Morse theory is that it enables us to determine the homotopy type of a simplicial complex using only the information of its \emph{critical simplices} corresponding to an acyclic matching. Furthermore, this theory allows us to compute the homology groups of a complex \ efficiently \cite{formanmain, gallai, vidit, viditpersistent, ask}. The computation of the homology groups of a complex reduces to the computation of the homology of a simpler chain complex (only critical simplices of corresponding dimensions generate chain groups).
	
	The \textit{independence complex} of a graph $G$, denoted by $\ig$, is a simplicial complex, generated by independent sets of $G$. Independence complexes (and their generalisations) of various families of graphs are important combinatorial objects, and they have been studied extensively in the literature \cite{barmak, deshpande, ehrenborg, engstromindependence, engstrom, kazuhirohomotopy, kazuhiro, anurag}. Several other important graph-theoretic complexes are associated with independence complexes. For example, the matching complex of a graph is the independence complex of its line graph, and the clique complex of a graph is the independence complex of its complementary graph, etc. 
	
	The techniques that are commonly used to study the topology of independence complexes in the literature involve sophisticated homotopy theoretical arguments; in particular, they rely heavily on important results due to Björner et al.\ \cite{bjorner, bjornerhomotopy}.
	
	In this paper, we construct an acyclic matching, recursively on the independence complex of a graph $G$ with a simplicial vertex, using a suitably chosen acyclic matching on the independence complex of some specific subgraphs of $G$ via the methods of discrete Morse theory. 
	
	First, we introduce a few useful notions and state the main theorem. A vertex is a \emph{universal vertex} if it is adjacent to all other vertices of $G$. For a vertex $v$ of $G$, $N_{G}(v)$ denotes the set of neighbours of $v$, and $N_{G}[v]=N_{G}(v) \sqcup \{v\}$. Let $f_{d}^{\V}$ denote the number of $d$-dimensional critical simplices in $\ig$, with respect to a given acyclic matching $\V$. 
	
	\begin{theorem}\label{mth}
		Let $v \in V(G)$ such that $\{v\} \subsetneq N_{G}[v] \subsetneq V(G)$ and $N_{G}[v]$ is a clique. Let $k$ be the number of universal vertices of the graph $G$. For all $u \in N_{G}(v)$, let $\vu$ be an acyclic matching on $\igu$. Then, there is an acyclic matching $\V$ on $\ig$ such that
		\begin{enumerate}[label=(\roman*)]
			\item ${f^{\V}_{0}}(\ig) = 1 + k$ = number of connected components of $\ig$,
			\item ${f^{\V}_{1}}(\mathcal{I}(G)) =\sum\limits_{u \in N_{G}(v)}
			{f^{\vu}_{0}}(\igu) - (|N_{G}(v)| - k),$
			\item for $t \geq 2$, ${f^{\V}_{t}}(\mathcal{I}(G))$ $ = \sum\limits_{u \in N_{G}(v)}{f^{\vu}_{t-1}}(\igu)$.
		\end{enumerate}
	\end{theorem}
	We observe that to construct an acyclic matching on $\ig$, by applying Theorem~\ref{mth} repeatedly, we require the existence of a simplicial vertex at each step after deleting the simplicial vertex and its neighbours, chosen in the previous step. A well-known and important family of graphs satisfying this property is the family of chordal graphs, as they are characterised by the \emph{perfect elimination ordering}. Another interesting example, arising from order theory, is the \emph{comparability graph of the $m \times n$ grid poset}, which is an example of a permutation graph and as well as a perfect graph.
	
	We use Theorem~\ref{mth} recursively and provide an algorithmic construction of an acyclic matching on $\ig$, where $G$ is a chordal graph (in Subsection~\ref{subsection chordal}) or belongs to a family $\g_{m,n}$ consisting of graphs obtained from the comparability graph of $m \times n$ grid poset by replacing each vertex by a clique of cardinality at least one (in Subsection~\ref{subsection comparabilty graph}), and obtain the following.
	
	\begin{theorem}\label{grid+chord}
		Let $G$ be a chordal graph or a graph in $\g_{m,n}$. Then $\ig$ admits an acyclic matching, where all $\V$-critical simplices, except possibly one $0$-dimensional simplex, are maximal.
	\end{theorem}
	Theorem~\ref{grid+chord} provides a required criterion to apply discrete Morse theoretic techniques developed in the subsequent section (see Theorem~\ref{sphmaxcrit} and Corollary~\ref{ind com sph with domination no}) that explicitly determine the homotopy type of the independence complexes of the aforementioned two classes of graphs. 
	
	We note that the \emph{power graph} of the cyclic group of order $p^mq^n$, for distinct primes $p$ and $q$, belongs to $\g_{m,n}$. 
	We mention here that the homotopy types of the independence complexes of chordal graphs and power graphs were obtained by the previously mentioned homotopy theoretical techniques in \cite{kazuhiro} and \cite{nn}, respectively.

	We remark that for the aforementioned two classes of graphs, the acyclic matching on the independence complex, constructed recursively starting from trivial base cases, is an optimal (i.e., one with the least number of critical simplices) one. The constructed acyclic matching also lets us establish the homotopy equivalence between the independence complex and a topological space with a simple structure (viz., a point or a wedge of spheres). However, it is noteworthy that, for a general graph, the homotopy type of its independence complex is not always ``nice'' or easily determinable. Also, obtaining an optimal acyclic matching on a complex is a computationally hard problem in general \cite{eugeciouglu, joswig, lewiner}. Theorem~\ref{mth} may still be applied to construct an efficient acyclic matching on the independence complex, especially when we start with efficient acyclic matchings on the base cases. From the computational aspect, this serves a useful purpose as it provides an important pre-processing framework for efficient homology computation via discrete Morse theory.
	\section{Preliminaries}\label{preli}
	Let $G$ be a (finite, simple, undirected) graph with the vertex set $V(G)$ and the edge set $E(G)$. The \emph{neighbourhood} $N_{G}(v)$ of a vertex $v$ in $G$ is the set of all vertices that are adjacent to $v$, and $N_{G}[v]=N_{G}(v) \sqcup \{v\}$. For $U \subseteq V(G)$, the graph $G - U$ is the subgraph of $G$ induced by the vertex set $V(G) \setminus U$. A subset $I$ of $V(G)$ is an \emph{independent set} if no two vertices of $I$ are adjacent. On the other hand, a \emph{clique} is a subset of $V(G)$ such that any two vertices in that subset are adjacent. A set $S \subseteq V(G)$ is a \emph{dominating set} if every vertex outside $S$ has at least one neighbour in $S$. The \emph{domination number} $\gamma(G)$ is the minimum cardinality of a dominating set in $G$.  A vertex $v$ is a \emph{simplicial vertex} if $N_{G}[v]$ is a clique.
	
	An \textit{(abstract) simplicial complex} (or simply, a complex), say $\mathcal{X}$, is a (finite) collection of finite sets such that 
	if $\beta \in \mathcal{X}$ and $\alpha \subseteq \beta $, then $\alpha \in \mathcal{X}$. Elements of $\mathcal{X}$ are called \emph{simplices} of $\mathcal{X}$.  The \emph{dimension} of a simplex $\alpha$ is $|\alpha| - 1$. If $\dim(\alpha)=d$, then we call $\alpha$ as \emph{$d$-dimensional simplex} (or, a \emph{$d$-simplex}) and denote it by $\alpha^{(d)}$. The \emph{dimension} of a simplicial complex $\mathcal{X}$ is \emph{$\dim(\mathcal{X})$} and is defined as $\max \{\dim (\alpha) \mid \alpha \in \mathcal{X}\}$. Let $f_{d}$ denote the number of $d$-dimensional simplices, and the \emph{$f$-vector} of $\mathcal{X}$ is the sequence ($f_{0}, f_{1}, \ldots, f_{\dim(\mathcal{X})}$). 
	
	We recall that an abstract simplicial complex may be realised as a topological space (unique up to homeomorphism) via its geometric realisation. While discussing the topological properties, we do not differentiate between a complex and its geometric realisation.
	
	The \emph{independence complex} $\ig$ of a graph $G$ is a simplicial complex consisting of all independent sets of $G$. We note that the $0$-simplex corresponding to a universal vertex in a graph $G$ is a maximal $0$-simplex in the independence complex $\ig$ (which corresponds to an isolated point of the space) and vice versa.
	
	\subsection{Basics of discrete Morse theory} \label{basic of dmt}
	
	Let $\mathcal{X}$ be a simplicial complex. The \emph{directed Hasse diagram} of $\mathcal{X}$ is the directed acyclic graph on $\mathcal{X}$ with edges from $\beta$ to $\alpha$, for each pair $\alpha, \beta$, with $\alpha \subsetneq \beta$ and $\dim(\beta)= \dim(\alpha)+1$. An \emph{acyclic matching} $\V$ on $\mathcal{X}$ is a matching in the directed Hasse diagram such that reversing all edges in $\V$ yields an acyclic directed graph. 
	
	More formally, a collection $\V=\{(\alpha, \beta) \mid \alpha \subseteq \beta, \dim(\beta)=\dim(\alpha) +1\}$ is a matching on $\mathcal{X}$, if each simplex appears in at most one pair of $\V$. For a matching $\V$, an \emph{$\V$-path} is a sequence of simplices
	\begin{center}{$\alpha_{0},\beta_{0},\alpha_{1},\beta_{1}, \ldots, \alpha_{r},\beta_{r},\alpha_{r+1}$},
	\end{center}
	where $\dim(\beta_{i})=\dim(\alpha_{j})+1$, for all $i$ and $j$, with $(\alpha_{i}, \beta_{i}) \in \V$ and $( \alpha_{i}\neq) \ \alpha_{i+1} \subsetneq \beta_{i}$, for all $i \in \{0,1, \ldots, r\}$. We observe that a matching $\V$ corresponds to an acyclic matching on $\mathcal{X}$ if every $\V$-path terminates without forming a (directed) cycle. A non-empty simplex $\alpha$ is \emph{$\V$-critical} (or simply, \emph{critical}) if it is not paired in $\V$ or it is paired with $\emptyset$. We note that a \emph{maximal} $0$-simplex of $\mathcal{X}$ is critical with respect to \emph{any} acyclic matching. 
	
	An acyclic matching is closely associated with the notion of \emph{collapse} in simple homotopy theory \cite{cohen}. Specifically, a \emph{collapsible complex} is equivalent to admitting an acyclic matching with a unique critical simplex (of dimension $0$).
	
	Let $f_{d}^{\V}$ denote the number of $d$-dimensional  critical simplices with respect to $\V$. The \emph{critical $f$-vector}, with respect to $\V$, is the sequence ($f_{0}^{\V}, f_{1}^{\V}, \ldots, f_{\dim(\mathcal{X})}^{\V}$). We note that, $f_{d}^{\V} \leq f_{d}$, for each $d \in \{0,1, \ldots , \dim{(\mathcal{X})}\}$.
	
	We recall that a CW-complex is a topological space, formed by attaching homeomorphic copies of \emph{balls} (called \emph{cells}) recursively by increasing dimension.
	We now state the fundamental theorem of discrete Morse theory.
	\begin{theorem} (Fundamental theorem of discrete Morse theory)\label{ftdmt}\cite{formanmain,formanuser}
		Let $\mathcal{X}$ be a simplicial complex with an acyclic matching $\V$. Then $\mathcal{X}$ is homotopy equivalent to a CW-complex $\mathcal{Y}$ (denoted by $\mathcal{X} \simeq \mathcal{Y}$) which contains $f_{d}^{\V}$ cells of dimension $d$, for each $d \geq 0$.

	\end{theorem}
	Theorem~\ref{ftdmt} implies that determining the topology of a complex would be easier if the number of critical simplices is reduced as much as possible. However, the problem of finding an acyclic matching that minimises the number of critical simplices is an \textsf{NP-hard} problem \cite{eugeciouglu, joswig, lewiner}.
	
	We recall that, for (pointed) topological spaces $X_{1}, X_{2}, \ldots, X_{n}$, the \emph{wedge sum} $\bigvee \limits_{i=1}^{n}X_{i}$ is the quotient space obtained by identifying the marked points (called \emph{base points}) of each $X_{i}$ to a single point. Also, we denote the $d$-dimensional sphere by $\mathbb{S}^{d}$, and abusing the notation, we denote the wedge of $t$ (copies of) $d$-dimensional spheres by $\bigvee \limits_{t} \mathbb{S}^{d}$.
	
	The following is a useful corollary of Theorem~\ref{ftdmt}.
	
	\begin{theorem}\label{sphth} \cite{formanuser}
		Let $\mathcal{X}$ be a (connected) simplicial complex with an acyclic matching $\V$. If \[f_{i}^{\V}=\begin{cases}
			1, & \text{if } i=0,\\
			t \text{ (}\geq 1\text{)}, & \text{if }i=d, \text{ for some } d\geq 1,\\
			0, & otherwise,
		\end{cases}\]
		then $\mathcal{X} \simeq \bigvee \limits_{t} \mathbb{S}^{d}.$
	\end{theorem}
	
	
	A well-known proof-of-concept application of the theorem above is on the simplicial complex consisting of the edge sets of all disconnected (labelled) graphs on the vertex set $\{1,2, \ldots, n\}$. This complex is of considerable importance in knot theory as it appears in Vassiliev's spectral sequence for computing the homology of the space of singular knots \cite{vassiliev}. The homotopy type of this complex is determined by constructing a suitable acyclic matching such that only critical simplices, other than a $0$-simplex, are the edge sets of an easily enumerable family of forests with two components \cite{formanuser,ampc} \cite[Chapter $11$]{kozlov}.

	
	Here we remark that, contrary to the intuition, a complex is not necessarily homotopy equivalent to a wedge of spheres of varying dimensions, even when there are no critical simplices in neighbouring dimensions of each critical simplex. A well-known counterexample \cite[Chapter $17$]{kozlov} uses the non-obvious topological fact that the Hopf map is homotopically non-trivial.
	
	In contrast, the subsequent theorem provides a criterion for homotopy equivalence to a wedge of spheres of varying dimensions. First, we introduce a generalisation of $\V$-paths. A \emph{generalised $\V$-path}, starting from a (critical) simplex $\tau_{0}$, is a sequence
	\[\tau_{0}, \sigma_{1}, \tau_{1}, \ldots, \sigma_{k},\tau_{k}, \sigma_{k+1},\]
	such that $(\sigma_{i},\tau_{i}) \in \V$, for all $i \in \{1,2, \ldots, k\}$ and $\tau_{i} \supsetneq \sigma_{i+1}$ ($\neq \sigma_{i}$), for all $i \in \{0,1, \ldots, k\}$.
	
	\begin{theorem} \label{koz theorem}\cite[Chapter $17$]{kozlov}
		Let $\mathcal{X}$ be a (connected) $d$-dimensional simplicial complex with an acyclic matching $\V$, having a unique critical $0$-simplex and at least one critical simplex of dimension $\geq 1$. Assume that for each critical simplex $\alpha$ of dimension $\geq 1$, the only possible critical simplices that are reachable from $\alpha$ via generalised $\V$-paths are $\alpha$ itself and the critical $0$-simplex. Then
		\[\mathcal{X} \simeq \bigvee_{f_{1}^{\V}} \mathbb{S}^{1} \bigvee_{f_{2}^{\V}} \mathbb{S}^{2} \cdots \bigvee_{f_{d}^{\V}} \mathbb{S}^{d}.\]
	\end{theorem}
	The following theorem provides a more specific yet more natural criterion compared to the one in Theorem~\ref{koz theorem}, which is easier to verify in practice.
	\begin{theorem}\label{sphmaxcrit}
		Let $\mathcal{X}$ be a $d$-dimensional simplicial complex with an acyclic matching $\V$. Suppose all the critical simplices are maximal, except possibly one
		critical $0$-simplex. Then either $\mathcal{X}$ is collapsible, or
		\[\mathcal{X} \simeq \bigvee\limits_{f_{0}^{\V}-1}\mathbb{S}^{0}\bigvee_{f_{1}^{\V}} \mathbb{S}^{1} \bigvee_{f_{2}^{\V}} \mathbb{S}^{2} \cdots \bigvee_{f_{d}^{\V}} \mathbb{S}^{d}.\]
	\end{theorem}
	\begin{proof}
		Suppose $\mathcal{X}$ contains only $0$-critical simplices. If $\mathcal{X}$ contains only one critical $0$-simplex, then $\mathcal{X}$ is collapsible. If $\mathcal{X}$ contains at least two $0$-dimensional critical simplices, then $\mathcal{X}$ is homotopy equivalent to $f_{0}^{\V}$ number of isolated points, which may be realised as the wedge of ($f_{0}^{\V}-1$) spheres of dimension $0$, that is
		\[\mathcal{X} \simeq \bigvee \limits_{f_{0}^{\V}-1}\mathbb{S}^{0}.\]
		
		Next, suppose $\mathcal{X}$ contains a critical simplex of dimension $\geq 1$. Without loss of generality, let $\emptyset$ be unpaired in $\V$. Let $\mathcal{X}'$ be the (connected) complex obtained by deleting all maximal $\V$-critical $0$-simplices (i.e., isolated points) from $\mathcal{X}$. We note that $\V$ is an acyclic matching on $\mathcal{X}'$.
		Let $\alpha \in \mathcal{X}'$ be a $\V$-critical simplex of dimension $\geq 1$. Since, all critical simplices of dimension $\geq 1$ are maximal, no generalised $\V$-path, starting from $\alpha$, reaches another critical simplex of dimension $\geq 1$. Thus, from Theorem~\ref{koz theorem}, it follows that
		\[\mathcal{X}' \simeq \bigvee_{f_{1}^{\V}} \mathbb{S}^{1} \bigvee_{f_{2}^{\V}} \mathbb{S}^{2} \cdots \bigvee_{f_{d}^{\V}} \mathbb{S}^{d}.\] Therefore, 
		\[\mathcal{X} \simeq \bigvee\limits_{f_{0}^{\V}-1}\mathbb{S}^{0}\bigvee_{f_{1}^{\V}} \mathbb{S}^{1} \bigvee_{f_{2}^{\V}} \mathbb{S}^{2} \cdots \bigvee_{f_{d}^{\V}} \mathbb{S}^{d}.\]
	\end{proof}
	In particular, when $\mathcal{X}$ is the independence complex of a graph with an acyclic matching $\V$, satisfying the property given in Theorem~\ref{sphmaxcrit}, since a maximal independent set of a graph is a dominating set, the dimension of each maximal critical simplex is at least $\gamma(G)-1$. Thus, we have the following corollary.
	
	\begin{corollary}\label{ind com sph with domination no}
		Let $\ig$ be the independence complex of a graph $G$ with an acyclic matching $\V$, such that all the critical simplices are maximal, except possibly one
		critical $0$-simplex. Then $\ig$ is collapsible or is homotopy equivalent to a wedge of spheres (of varying dimensions), such that each sphere is of dimension at least $ \gamma(G)-1$.
	\end{corollary}
	Here we remark that although the homology of a simplicial complex satisfying the criterion of Theorem~\ref{sphmaxcrit} can be readily inferred from its homotopy type, the task of homology computation itself, using discrete Morse theory, is also trivial, as each boundary operator of the chain complex becomes trivial.   
	
	\section{Construction of an acyclic matching on the independence complex of a graph with a simplicial vertex}\label{construction of ig}
	In this section, we construct an acyclic matching $\V$ on $\ig$ of a graph $G$ with a simplicial vertex and also provide a single-step recursive count of the $\V$-critical simplices of $\ig$.
	
	\begin{observation}\label{obsmth}
		Let $v \in V(G)$ be such that $N_{G}(v)$ is a clique.
		\begin{enumerate}
			\item For $u \in N_{G}(v)$, if $\alpha \in \igu$, then $\alpha \in \igv$,
			that is \[\bigcup \limits_{u \in N_{G}(v)}\igu \subseteq \igv.\]
			
			\item $\ig$ can be partitioned as follows.
			\begin{align*}
				\ig =&\bigcup_{u \in N_{G}(v)} \igu\\
				&\bigsqcup \left(\igv \setminus (\bigcup_{u \in N_{G}(v)} \igu )\right)\\
				&\bigsqcup_{u \in N_{G}(v)} \{\alpha \cup\{u\} \mid \alpha  \in \igu\} \\
				&\bigsqcup \{\alpha \cup\{v\} \mid \alpha  \in \igv \}.
			\end{align*}
			\item The simplicial complex obtained from $\ig$ after deleting the maximal $0$-simplices (i.e., isolated points) is connected.
			
		\end{enumerate}
	\end{observation}
	
	\begin{lemma}\label{lmiso}
		Let $v \in V(G)$ be an isolated vertex of $G$. Then, there is an acyclic matching on $\ig$ such that the only critical simplex is $\{v\}$ (consequently, $\ig$ is collapsible). 
	\end{lemma}
	\begin{proof}
		
		We may verify that $\V= \{(\alpha, \alpha \cup \{v\}) \mid \alpha \in \igv, v \notin \alpha\}$ is an acyclic matching on $\ig$ such that $\{v\}$ is the only $\V$-critical simplex.
	\end{proof}
	\begin{observation}\label{lmcom}
		Let $G$ be a complete graph. Then, with respect to any acyclic matching on $\ig$, each non-empty simplex is critical, maximal, and is of dimension $0$.
	\end{observation}
	We now proceed to prove the main theorem of this article.
	\begin{proof}[Proof of Theorem~\ref{mth}]
		Let $v \in V(G)$ such that $v$ is a simplicial vertex. Let $\vu$ be an acyclic matching on $\igu$, for all $u \in N_{G}(v)$. We construct an acyclic matching $\V$ on $\mathcal{I}(G)$ as follows.
		\begin{enumerate}[label=(\roman*)]
			\item For each pair $(\alpha, \beta) \in \vu$ (with $\alpha \neq \emptyset$), we  pair $\alpha \cup\{u\}$ with $\beta \cup \{u\}$ in $\V$ (i.e., $(\alpha \cup \{u\}, \beta \cup \{u\}) \in \V$), for all $u \in N_{G}(v).$
			\item Furthermore, for all $u \in N(v)$ such that $G - N_{G}[u] \neq \emptyset $, we pair $\{u\}$ in $\V$ as follows. For such a vertex $u$, since $\igu \neq \{\emptyset\}$, there is at least one $\vu$-critical $0$-simplex in $\igu$. We choose and fix a critical $0$-simplex $\{x_{u}\} \in \igu$. We pair $\{u\}$ with $\{u,x_{u}\}$ in $\V$ (i.e., $(\{u\},\{u, x_{u}\}) \in \V)$, for all such $u$.
			
			\item Finally, for all $\alpha \in \igv$, we add $(\alpha, \alpha \cup \{v\})$ to $ \V$.
		\end{enumerate}
		
		From Observation~\ref{obsmth}, it follows that each simplex in $\mathcal{I}(G)$ is either uniquely paired in $\V$ or left unpaired. Therefore, the matching $\V$ is well-defined. 
		Next, we prove the acyclicity of $\V$.
		
		We observe that the only possible (maximal) $\V$-path, starting from $\{u\}$, is \[\{u\}, \{u, x_{u}\}, \{x_{u}\}, \{x_{u},v\}, \{v\},\] and the only possible (maximal) $\V$-path, starting from $\{w\}$ ($\neq \{u\}$), is \[\{w\}, \{w,v\}, \{v\},\] none of which extends to a cycle.
		
		Let, if possible, \[\alpha_{0}^{(d)},\beta_{0}^{(d+1)},\alpha_{1}^{(d)},\beta_{1}^{(d+1)}, \ldots,  \alpha_{r+1}^{(d)}=\alpha_{0}^{(d)}, \text{ (}d \geq 1\text{)}\] be an $\V$-path forming a cycle. Suppose, for all $i \in \{0,1, \ldots ,r+1\}$, $\alpha_{i}$ contains $u$. Then $ u \in \beta_{i}$, for all $i \in \{0,1, \ldots ,r\}$. This implies that the corresponding $\vu$-path obtained by 
		removing $u$ from each simplex, that is
		\[
		{(\alpha_{0} \setminus \{u\})}^{(d-1)},{(\beta_{0} \setminus \{u\} )}^{(d)},{(\alpha_{1} \setminus \{u\})}^{(d-1)}, \ldots, {(\beta_{r} \setminus \{u\})}^{(d)}, {(\alpha_{r+1} \setminus \{u\})}^{(d-1)}\] 
		forms a cycle, which contradicts the acyclicity of $\vu$.
		
		Now, suppose there exists a least $t \in \{0,1, \ldots, r\}$ such that $u \notin \alpha_{t}.$ Then $\beta_{t} = \alpha_{t} \sqcup \{v\}$ and $\alpha_{t+1}$ must contain $v$. By construction, this $\V$-path cannot be extended further. So this is a contradiction.
		Hence, the matching $\V$ is acyclic.
		
		We now characterise all the $\V$-critical simplices.
		\begin{enumerate}[label=(\roman*)]
			\item  The $0$-simplex $\{v\}$ is critical with respect to $\V$. Also, for all universal vertex $u \in V(G)$ (thus, $u \in N_{G}(v)$), the $0$-simplex $\{u\}$ is $\V$-critical. 
			\item $\V$-critical $1$-simplices are $\{w,u\}$, where $\{w\} \in \igu$, with $w \neq x_{u}$, is a $\vu$-critical $0$-simplex.
			
			\item All other $\V$-critical simplices (of dimension $\geq 2$) of $\ig$ are of the type $\alpha \sqcup \{u\}$ for each $\vu$-critical simplex $\alpha$ (of dimension $\geq 1$)  of $\igu$.
		\end{enumerate}
		The theorem follows from the count of critical simplices (of each dimension), from the characterisation above. 
	\end{proof}
	
	\subsection{Independence complex of chordal graphs}\label{subsection chordal}
	A graph is a \emph{chordal graph} if there is no induced cycle of length greater than or equal to $4$. A \emph{perfect elimination ordering} of a graph is an ordering on its vertices $v_{1},v_{2}, \ldots, v_{n}$ such that, for each vertex $v_{i}$, the neighbours of $v_{i}$ that appear after $v_{i}$ in the ordering, form a clique.
	The following theorem characterises chordal graphs.
	\begin{theorem} \cite{dirac,fulkerson, rose} \label{PEO}
		A graph $G$ is chordal if and only if it admits a perfect elimination ordering.
	\end{theorem}
	This characterisation lets us apply Theorem~\ref{mth} recursively and get the following.
	\begin{theorem}\label{chordalth}
		Let $G$ be a chordal graph. Then there is an acyclic matching $\V$ on $\ig$ such that all $\V$-critical simplices, except possibly one $0$-simplex, are maximal.
	\end{theorem}
	\begin{proof}
		We use induction on $|V(G)|$. Consider $|V(G)|=1$. Then, the only possible graph is $K_{1}$, the complete graph with $1$ vertex. Thus, with respect to any acyclic matching $\V$ on $\ii({K_{1}})$, there is only one critical $0$-simplex. 
		
		Now consider a chordal graph $G$ with vertex set $V(G)$ ($ |V(G)| \geq 2$). If $G$ contains isolated vertices, then from Lemma~\ref{lmiso}, there exists an acyclic matching on $\ig$ such that there is a unique critical simplex (of dimension $0$). On the other hand, if $G$ is complete, then from Observation~\ref{lmcom}, for any acyclic matching on $\ig$, each non-empty simplex is critical, maximal, and of dimension $0$.
		
		Let $G$ be a non-complete graph without any isolated vertices. Assume that the statement holds for all chordal graphs $H$ with $|V(H)| < |V(G)|$. We note that the first vertex, say $v$, in a perfect elimination ordering of $G$, is a simplicial and non-universal vertex. Let $u \in N_{G}(v)$ and thus $G - N_{G}[u]$ is also a chordal graph. By the induction hypothesis, there exists an acyclic matching $\vu$ on $\igu$ such that all the $\vu$-critical simplices are maximal, except possibly one $0$-simplex. 
		
		We use all such $\vu$ to construct an acyclic matching $\V$ on $\ig$, following the proof of Theorem~\ref{mth}. If there is a \emph{non-maximal} $\vu$-critical $0$-simplex, say $\{x\}$, then we set $x_{u} = x$ (i.e., $\{u\}$ is paired with $\{x,u\}$ in $\V$). We have the following.
		\begin{enumerate}[label=(\roman*)]
			\item $\{v\}$ is a critical $0$-simplex, and all other critical $0$-simplices correspond to universal vertices of $G$.
			\item $\V$-critical simplices, of dimension $\geq 1$, are of the form $\alpha \sqcup \{u\}$, for each $\vu$-critical simplex $\alpha \in \igu$, other than the possible non-maximal $0$-simplex. 
		\end{enumerate}
		We note that all $\V$-critical simplices other than $\{v\}$ are maximal.
	\end{proof}
	The following result, describing the homotopy type of the independence complex of a chordal graph, follows from Corollary~\ref{ind com sph with domination no} and Theorem~\ref{chordalth}. However, a procedure based on homotopy-theoretic techniques for determining its homotopy type appears in \cite{kazuhiro}. 
	\begin{corollary}
		Let $G$ be a chordal graph. Then, either $\ig$ is collapsible or $\ig$ is homotopy equivalent to a wedge of spheres (of varying dimensions), such that the dimension of each sphere is at least $ \gamma(G)-1$.
	\end{corollary}

	\subsection{Independence complex of comparability graphs of grid posets}\label{subsection comparabilty graph}
	
	As mentioned previously, instead of the comparability graph of the $m \times n$ grid poset, we consider a more general family $\g_{m,n}$. For $m,n \in \mathbb{N} \cup \{0\}$, $\g_{m,n}$ is the class of graphs with the vertex set $V = \bigsqcup \limits_{i=0}^{m} \bigsqcup \limits_{j=0}^{n} V_{i,j}$, where each $V_{i,j} \neq \emptyset$, and a pair of vertices $x \in V_{i_{1},j_{1}}$ and $y \in V_{i_{2},j_{2}}$ are adjacent if 
	\begin{enumerate}[label=(\roman*)]
		\item $i_{1} \leq i_{2}$ and $j_{1} \leq j_{2}$, or
		
		\item $i_{1} \geq i_{2}$ and $j_{1} \geq j_{2}$ (see Figure~\ref{fig griddefn}).
	\end{enumerate}

	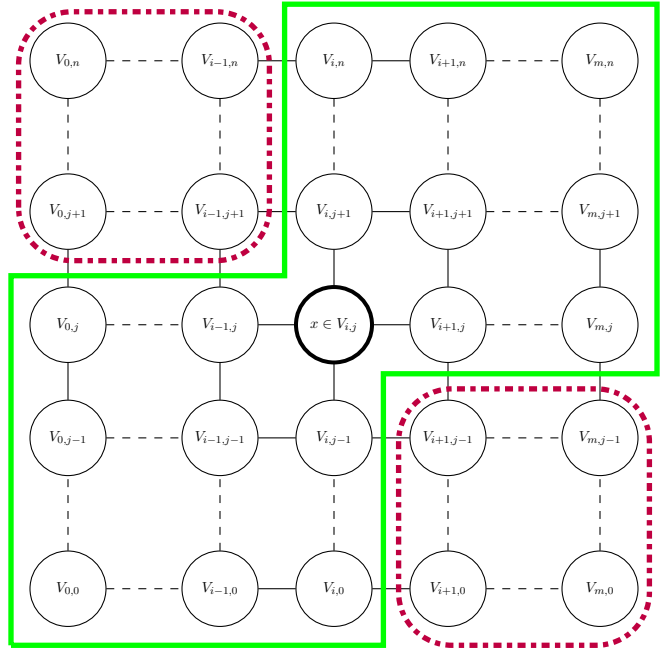
\begin{figure}[htbp]
		\begin{center}
			
			\begin{tikzpicture}[transform shape, scale = 0.5]

				\node[circle, draw, minimum size=20mm] (a) at (0,0) {$V_{0,0}$};
				\node[circle, draw, minimum size=20mm] (b) at (4,0) {$V_{i-1,0}$};
				\node[draw, circle, minimum size=20mm] 
				(c) at (7,0) {$ V_{i,0}$};
				
				\node[circle, draw, minimum size=20mm] (d) at (10,0) {$V_{i+1,0}$};
				\node[circle, draw, minimum size=20mm] (e) at (14,0) {$V_{m,0}$};

				\node[circle, draw, minimum size=20mm] (A) at (0,4) {$V_{0,j-1}$};
				\node[circle, draw, minimum size=20mm] (B) at (4,4) {$ V_{i-1,j-1}$};
				\node[circle, draw, minimum size=20mm] (C) at (7,4) {$V_{i,j-1}$};
				
				\node[circle, draw, minimum size=20mm] (D) at (10,4) {$ V_{i+1,j-1}$};
				\node[circle, draw, minimum size=20mm] (E) at (14,4) {$V_{m,j-1}$};

				\node[circle, draw, minimum size=20mm] (X) at (0,7) {$V_{0,j}$};
				\node[circle, draw, minimum size=20mm] (Y) at (4,7) {$V_{i-1,j}$};
				\node[circle, draw, minimum size=20mm, line width=1.5pt] (Z) at (7,7) {$x \in V_{i,j}$};
				
				\node[circle, draw, minimum size=20mm] (R) at (10,7) {$V_{i+1,j}$};
				\node[circle, draw, minimum size=20mm] (S) at (14,7) {$V_{m,j}$};

				\node[circle, draw, minimum size=20mm] (k) at (0,10) {$V_{0,j+1}$};
				
				\node[circle, draw, minimum size=20mm] (l) at (4,10) {$V_{i-1,j+1}$};
				\node[circle, draw, minimum size=20mm] (m) at (7,10) {$V_{i,j+1}$};
				\node[circle, draw, minimum size=20mm] (n) at (10,10) {$V_{i+1,j+1}$};
				\node[circle, draw, minimum size=20mm] (o) at (14,10) {$V_{m,j+1}$};

				\node[circle, draw, minimum size=20mm] (G) at (0,14) {$V_{0,n}$};
				
				\node[circle, draw, minimum size=20mm] (H) at (4,14) {$V_{i-1,n}$};
				\node[circle, draw, minimum size=20mm] (I) at (7,14) {$V_{i,n}$};
				
				\node[circle, draw, minimum size=20mm] (J) at (10,14) {$V_{i+1,n}$};
				\node[circle, draw, minimum size=20mm] (K) at (14,14) {$V_{m,n}$};

				
				\draw[dashed] (a) -- (b);
				\draw (b) -- (c);
				\draw (c) -- (d);
				\draw[dashed] (d) -- (e);

				\draw[dashed] (A) -- (B);
				\draw (B) -- (C);
				\draw (C) -- (D);
				\draw[dashed] (D) -- (E);

				\draw[dashed] (X) -- (Y);
				\draw (Y) -- (Z);
				\draw (Z) -- (R);
				\draw[dashed] (R) -- (S);

				\draw[dashed] (k) -- (l);
				\draw (l) -- (m);
				\draw (m) -- (n);
				\draw[dashed] (n) -- (o);
				
				
				\draw[dashed] (G) -- (H);
				\draw (H) -- (I);
				\draw (I) -- (J);
				\draw[dashed] (J) -- (K);

				\draw[dashed] (a) -- (A);
				\draw[dashed] (b) -- (B);
				\draw[dashed] (c) -- (C);
				\draw[dashed] (d) -- (D);
				\draw[dashed] (e) -- (E);

				\draw (A) -- (X);
				\draw (B) -- (Y);
				\draw (C) -- (Z);
				\draw (D) -- (R);
				\draw (E) -- (S);

				\draw (X) -- (k);
				\draw (Y) -- (l);
				\draw (Z) -- (m);
				\draw (R) -- (n);
				\draw (S) -- (o);
				
				\draw[dashed] (G) -- (k);
				\draw[dashed] (H) -- (l);
				\draw[dashed] (I) -- (m);
				\draw[dashed] (J) -- (n);
				\draw[dashed] (K) -- (o);

				\draw[color=green, line width=2pt] (-1.5, -1.5) -- (-1.5,8.3) --(5.7,8.3) -- (5.7, 15.5) -- (15.5,15.5) -- (15.5,5.7) -- (8.3,5.7) -- (8.3,-1.5) -- (-1.5, -1.5);
				
				\draw[line width=2pt, color=purple, rounded corners=20pt, dashdotted] (-1.3,8.7) rectangle (5.3,15.3);
				\draw[line width=2pt, color=purple, rounded corners=20pt, dashdotted] (8.7,-1.5) rectangle (15.3,5.3);

			\end{tikzpicture} 
		\end{center}
		\caption{A graph $G \in \g_{m,n}$ is represented by this grid diagram, where two vertices are adjacent if and only if they are reachable via a (north-east) lattice path. For instance, the region bounded by the solid line segments represents $N_{G}[x]$, for $x \in V_{i,j}$, and the dotted rectangles contain the non-neighbours of $ x$.} \label{fig griddefn}
	\end{figure}
	
	\begin{observation}
		Let $G \in \g_{m,n}$. Then, the following hold.
		\begin{enumerate}
			\item For all $i,j$, the set $V_{i,j}$ is a clique.
			\item All the vertices in $V_{0,0}$ and $V_{m,n}$ are universal vertices in $G$. Moreover, if $m\geq 1$ and $n \geq 1$, then there are no universal vertices other than vertices in $V_{0,0} \cup V_{m,n}$.  
			
		\end{enumerate}
	\end{observation}
	
	\begin{theorem}\label{gridth}
		For any $G \in \g_{m,n}$, there is an acyclic matching $\V$ on $\ig$ such that all $\V$-critical simplices, except possibly one $0$-simplex, are maximal. As a consequence, \normalfont{(if $|V(G)|\geq 2$)}
		\[\ii(G) \simeq \bigvee_{f_{0}^{\V}-1} \mathbb{S}^{0} \bigvee_{f_{1}^{\V}} \mathbb{S}^{1}
		\bigvee_{f_{2}^{\V}} \mathbb{S}^{2}
		\cdots \bigvee_{f_{d}^{\V}} \mathbb{S}^{d},\]
		where $d=\dim(\ig)$ ($=\min\{m,n\}$).
		
	\end{theorem}
	\begin{proof}
		We use induction on $m$ and $n$. Let $G \in \g_{m',0}$, where $0 \leq m' \leq m$. Then $G$ is a complete graph with $V(G)$ $=\bigsqcup \limits_{i=0}^{m'} V_{i,0}$. Thus, all non-empty simplices of $\ig$ are of dimension $0$. Therefore, it follows from Observation~\ref{lmcom}, with respect to any acyclic matching $\V$ on $\ig$, each non-empty simplex is critical and maximal.
		
		Similarly, $G \in \g_{0,n'}$ ($1 \leq n' \leq n$) is also a complete graph. Thus, with respect to any acyclic matching $\V$ on $\ig$, all non-empty simplices are critical, maximal, and $0$-dimensional.
		
		Let $m \geq 1$ and $n \geq 1$, and let the statement hold for any graph in $\g_{m',n'}$, for $m' < m$ and $ n' < n$. 
		We observe that for $G \in \g_{m,n}$, any $v \in V_{m,0}$ (or, $v \in V_{0,n}$) is \emph{not} a universal vertex. Also, $N_{G}[v]=(\bigsqcup \limits_{i=0}^{m}V_{i,0}) \cup (\bigsqcup \limits_{j=1}^{n} V_{m,j})$ is a clique.
		
		If $u \in N(v)$ is \emph{not} a universal vertex, then either $u \in V_{i,0}$, with $1 \leq i \leq m $, or $u \in V_{m,j}$, with $1 \leq j \leq n-1$.
		If $u \in V_{i,0}$  with $1 \leq i \leq m $ (or, $u \in V_{m,j}$ with $1 \leq j \leq n-1$), then $G - N_{G}[u]$ can be realised as a member of $\g_{i-1,n-1}$ (or, $\g_{m-1,n-j-1}$) (see Figure~\ref{fig g-nu graph}).
		
		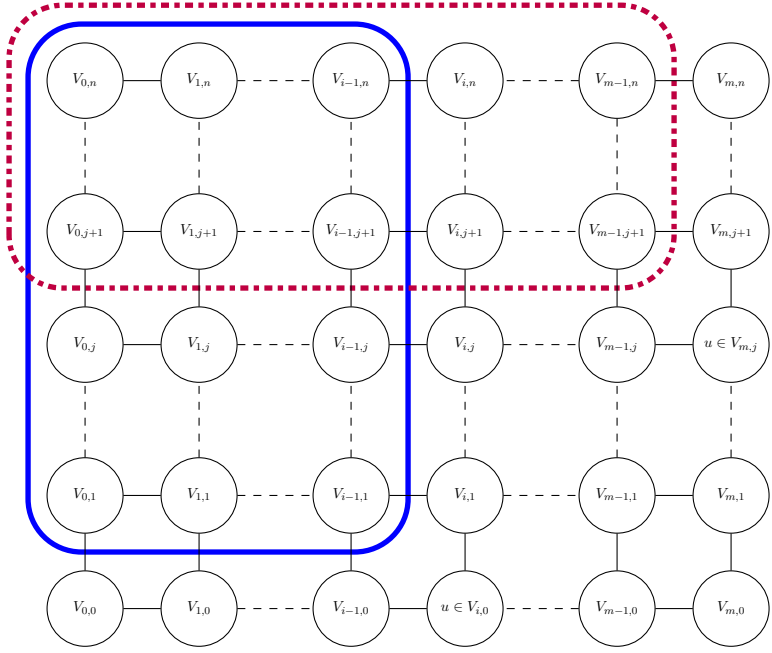
\begin{figure}[htbp]
			\begin{center}
				
				\begin{tikzpicture}[transform shape, scale = 0.5]
					
					\node[circle, draw, minimum size=20mm] (a) at (0,0) {$V_{0,0}$};
					\node[circle, draw, minimum size=20mm] (b) at (3,0) {$V_{1,0}$};
					\node[circle, draw, minimum size=20mm] (c) at (7,0) {$V_{i-1,0}$};
					\node[draw, circle, minimum size=20mm] 
					(p) at (10,0) {$u \in V_{i,0}$};
					
					\node[draw, circle, minimum size=20mm] (d) at (14,0) {$V_{m-1,0}$};
					
					\node[circle, draw, minimum size=20mm] (e) at (17,0) {$V_{m,0}$};
					
					\node[circle, draw, minimum size=20mm] (f) at (0,3) {$V_{0,1}$};
					\node[circle, draw, minimum size=20mm] (g) at (3,3) {$V_{1,1}$};
					\node[circle, draw, minimum size=20mm] (h) at (7,3) {$V_{i- 1,1}$};
					\node[circle, draw, minimum size=20mm] (q) at (10,3) {$V_{i,1}$};
					
					\node[circle, draw, minimum size=20mm] (i) at (14,3) {$V_{m-1,1}$};
					\node[circle, draw, minimum size=20mm] (j) at (17,3) {$V_{m,1}$};

					\node[circle, draw, minimum size=20mm] (t) at (0,7) {$V_{0,j}$};
					\node[circle, draw, minimum size=20mm] (u) at (3,7) {$V_{1,j}$};
					\node[circle, draw, minimum size=20mm] (v) at (7,7) {$V_{i-1,j}$};
					\node[circle, draw, minimum size=20mm] (w) at (10,7) {$V_{i,j}$};
					
					\node[circle, draw, minimum size=20mm] (x) at (14,7) {$V_{m-1,j}$};
					\node[circle, draw, minimum size=20mm] (y) at (17,7) {$u \in V_{m,j}$};
					
					\node[circle, draw, minimum size=20mm] (A) at (0,10) {$V_{0,j+1}$};
					\node[circle, draw, minimum size=20mm] (B) at (3,10) {$V_{1,j+1}$};
					\node[circle, draw, minimum size=20mm] (C) at (7,10) {$V_{i-1,j+1}$};
					\node[circle, draw, minimum size=20mm] (D) at (10,10) {$V_{i,j+1}$};
					\node[circle, draw, minimum size=20mm] (F) at (14,10) {$V_{m-1,j+1}$};
					\node[circle, draw, minimum size=20mm] (E) at (17,10) {$V_{m,j+1}$};
					
					\node[circle, draw, minimum size=20mm] (k) at (0,14) {$V_{0,n}$};
					\node[circle, draw, minimum size=20mm] (l) at (3,14) {$V_{1,n}$};
					\node[circle, draw, minimum size=20mm] (m) at (7,14) {$V_{i-1,n}$};
					\node[circle, draw, minimum size=20mm] (r) at (10,14) {$V_{i,n}$};
					\node[circle, draw, minimum size=20mm] (n) at (14,14) {$V_{m-1,n}$};
					\node[circle, draw, minimum size=20mm] (o) at (17,14) {$V_{m,n}$};
					
					\draw[line width=2pt, color=blue, rounded corners=20pt] (-1.5,1.5) rectangle (8.5,15.5);
					\draw[line width=2pt, color=purple, dashdotted , rounded corners=20pt] (-2,8.5) rectangle (15.5,16);
					
					\draw (a) -- (b);
					\draw[dashed] (b) -- (c);
					\draw (c) -- (p);
					\draw[dashed] (d) -- (p);
					\draw (e) -- (d);
					
					\draw (a) -- (f);
					\draw (b) -- (g);
					\draw (c) -- (h);
					\draw (d) -- (i);
					\draw (e) -- (j);
					\draw(p) -- (q);

					\draw (f) -- (g);
					\draw[dashed] (g) -- (h);
					\draw (h) -- (q);
					
					\draw[dashed] (q) -- (i);
					\draw (i) -- (j);
					
					\draw (t) -- (u);
					\draw[dashed] (u) -- (v);
					\draw (v) -- (w);
					\draw[dashed] (w) -- (x);
					\draw (x) -- (y);
					
					\draw (A) -- (B);
					\draw[dashed] (B) -- (C);
					\draw (C) -- (D);
					\draw[dashed] (D) -- (F);
					\draw (F) -- (E);

					\draw (k) -- (l);
					\draw[dashed] (l) -- (m);
					\draw (m) -- (r);
					\draw[dashed] (n) -- (r);
					
					\draw (n) -- (o);
					
					\draw[dashed] (f) -- (t);
					\draw[dashed] (g) -- (u);
					\draw[dashed] (h) -- (v);
					\draw[dashed] (q) -- (w);
					\draw[dashed] (i) -- (x);
					\draw[dashed] (j) -- (y);
					
					\draw (t) -- (A);
					\draw (u) -- (B);
					\draw (v) -- (C);
					\draw (w) -- (D);
					\draw (F) -- (x);
					\draw (y) -- (E);
					
					\draw[dashed] (A) -- (k);
					\draw[dashed] (B) -- (l);
					\draw[dashed] (C) -- (m);
					\draw[dashed] (D) -- (r);
					\draw[dashed] (n) -- (F);
					\draw[dashed] (E) -- (o);
					
				\end{tikzpicture} 
			\end{center}
			\caption{For $u \in V_{i,0}$, the solid rectangle represents the graph $G - N_{G}[u] \in \g_{i-1,n-1}$ with the vertex set $\bigsqcup \limits_{r=0}^{i-1} \bigsqcup \limits_{s=0}^{n-1}V'_{r,s}$, where $V'_{r,s}=V_{r,s+1}$, and for $u \in V_{m,j}$, the dotted rectangle represents the graph $G - N_{G}[u] \in \g_{m-1,n-j-1}$ with the vertex set$\bigsqcup \limits_{r=0}^{m-1} \bigsqcup \limits_{s=0}^{n-j-1}V'_{r,s}$, where $V'_{r,s}=V_{r,s+j+1}$.}\label{fig g-nu graph}
		\end{figure}
		
		By the induction hypothesis, there is an acyclic matching $\vu$ on $\igu$ with all the $\vu$-critical simplices being maximal, except possibly one $0$-simplex. 
		
		We construct an acyclic matching $\V$ on $\ig$ by using all such $\vu$, following the proof of Theorem~\ref{mth}. If a \emph{non-maximal} $\vu$-critical $0$-simplex exists, say $\{x\}$, we set $x_{u} = x$ (i.e., ($\{u\},\{x,u\}$)$ \in \V$). Thus,
		\begin{enumerate}[label=(\roman*)]
			\item $\{v\}$ is $\V$-critical, and all other $\V$-critical $0$-simplices correspond to universal vertices of $G$.
			\item $\V$-critical simplices, of dimension $\geq 1$, are of the form $\alpha \sqcup \{u\}$, for each $\vu$-critical maximal (in $\igu$) simplex $\alpha$.
		\end{enumerate}
		We note that all $\V$-critical simplices other than $\{v\}$ are maximal.
		
		Moreover, from Theorem~\ref{sphmaxcrit}, it follows that for an acyclic matching, constructed above, $\ig$ is homotopy equivalent to the wedge of ($f_{d}^{\V}-\delta_{d,0}$) spheres of dimension $d$ (i.e., $\mathbb{S}^{d}$), where $d\geq 0$.
	\end{proof}

	Given the full recursive construction of an acyclic matching $\V$ on $\ig$ established in the proof of Theorem~\ref{gridth}, a complete formula for $f_{d}^{\V}(\ig)$ follows naturally by iteratively applying the single-step recursion from Theorem~\ref{mth}. The subsequent proposition formalises this result, providing an explicit recursion formula for the critical $f$-vector. We note that an analogous formula for the numbers of spheres of each dimension for the independence complex of the power graph of the cyclic group of order $p^mq^n$, for distinct primes $p$ and $q$, is obtained in \cite{nn}.

	Let $G \in \g_{m,n}$ and $\dim(\ig)=d$ ($=\min\{m,n\}$). Let $\V$ be an acyclic matching on $\ig$, as constructed recursively in the proof of Theorem~\ref{gridth}.
	For $i \in \{0,1, \ldots,m-1\}$, $j \in \{1,2, \ldots,n\}$, let $G_{i,j}$ be the  subgraph of $G$ induced by the vertex set 
	$\bigsqcup \limits_{r=0}^{i} \bigsqcup \limits_{s=j}^{n}V_{r,s}$. Then $G_{i,j} \in \g_{i, n-j}$, with $V(G
	_{i,j})=\bigsqcup \limits_{r=0}^{i} \bigsqcup \limits_{s=0}^{n-j}V'_{r,s}$, where $V'_{r,s}=V_{r,s+j}$ (see Figure~\ref{fig g_ij graph}). 
	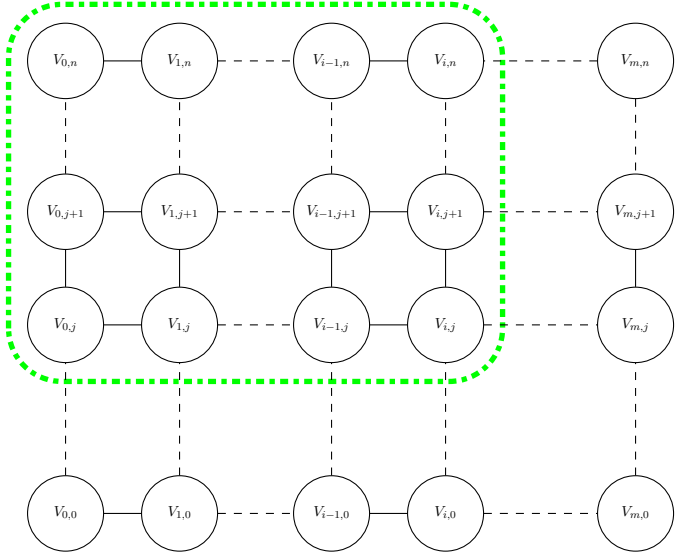
\begin{figure}[htbp]
		\begin{center}
			
			\begin{tikzpicture}[transform shape, scale = 0.5]

				\node[circle, draw, minimum size=20mm] (f) at (0,1) {$V_{0,0}$};
				\node[circle, draw, minimum size=20mm] (g) at (3,1) {$V_{1,0}$};
				\node[circle, draw, minimum size=20mm] (h) at (7,1) {$V_{i- 1,0}$};
				\node[circle, draw, minimum size=20mm] (q) at (10,1) {$V_{i,0}$};
				
				\node[circle, draw, minimum size=20mm] (i) at (15,1) {$V_{m,0}$};

				\node[circle, draw, minimum size=20mm] (t) at (0,6) {$V_{0,j}$};
				\node[circle, draw, minimum size=20mm] (u) at (3,6) {$V_{1,j}$};
				\node[circle, draw, minimum size=20mm] (v) at (7,6) {$V_{i-1,j}$};
				\node[circle, draw, minimum size=20mm] (w) at (10,6) {$V_{i,j}$};
				
				\node[circle, draw, minimum size=20mm] (x) at (15,6) {$V_{m,j}$};
				
				\node[circle, draw, minimum size=20mm] (A) at (0,9) {$V_{0,j+1}$};
				\node[circle, draw, minimum size=20mm] (B) at (3,9) {$V_{1,j+1}$};
				\node[circle, draw, minimum size=20mm] (C) at (7,9) {$V_{i-1,j+1}$};
				\node[circle, draw, minimum size=20mm] (D) at (10,9) {$V_{i,j+1}$};
				\node[circle, draw, minimum size=20mm] (F) at (15,9) {$V_{m,j+1}$};
				
				\node[circle, draw, minimum size=20mm] (k) at (0,13) {$V_{0,n}$};
				\node[circle, draw, minimum size=20mm] (l) at (3,13) {$V_{1,n}$};
				\node[circle, draw, minimum size=20mm] (m) at (7,13) {$V_{i-1,n}$};
				\node[circle, draw, minimum size=20mm] (r) at (10,13) {$V_{i,n}$};
				\node[circle, draw, minimum size=20mm] (n) at (15,13) {$V_{m,n}$};
				
				\draw[line width=2pt, color=green, rounded corners=20pt, dashdotted] (-1.5,4.5) rectangle (11.5,14.5);
				
				\draw (f) -- (g);
				\draw[dashed] (g) -- (h);
				\draw (h) -- (q);
				
				\draw[dashed] (q) -- (i);

				\draw (t) -- (u);
				\draw[dashed] (u) -- (v);
				\draw (v) -- (w);
				\draw[dashed] (w) -- (x);

				\draw (A) -- (B);
				\draw[dashed] (B) -- (C);
				\draw (C) -- (D);
				\draw[dashed] (D) -- (F);

				\draw (k) -- (l);
				\draw[dashed] (l) -- (m);
				\draw (m) -- (r);
				\draw[dashed] (n) -- (r);

				\draw[dashed] (f) -- (t);
				\draw[dashed] (g) -- (u);
				\draw[dashed] (h) -- (v);
				\draw[dashed] (q) -- (w);
				\draw[dashed] (i) -- (x);

				\draw (t) -- (A);
				\draw (u) -- (B);
				\draw (v) -- (C);
				\draw (w) -- (D);
				\draw (F) -- (x);
				
				\draw[dashed] (A) -- (k);
				\draw[dashed] (B) -- (l);
				\draw[dashed] (C) -- (m);
				\draw[dashed] (D) -- (r);
				\draw[dashed] (n) -- (F);

			\end{tikzpicture} 
		\end{center}
		\caption{For $ i \in \{0,1, \ldots, m-1\}$ and $j \in \{1, 2, \ldots, n\}$ the dotted rectangle represents the graph $G_{i,j}$.} \label{fig g_ij graph}
	\end{figure}

	Let $d_{i,j}:=\dim(\ii(G_{i,j}))= \min \{i,n-j\}$.  For $\ell \in \{0,1, \ldots,d_{i,j}\}$, let $c_{i,j}^{(\ell)}$ denote the number of $\ell$-dimensional critical simplices of $\ii(G_{i,j})$ with respect to an acyclic matching $\V'$, constructed recursively in the proof of Theorem~\ref{gridth}. In other words, $c_{i,j}^{(\ell)}=f_{\ell}^{\V'}(\ii(G_{i,j}))$. We provide a recursion formula for $c_{i,j}^{(\ell)}$ and the critical $f$-vector of $\ig$ (with respect to $\V$) as follows. 
	
	\begin{proposition}\label{gridcountth}
		For $i \in \{0,1, \ldots,m-1\}$, $j \in \{1,2, \ldots,n\}$, $c_{i,j}^{(\ell)}$ satisfies the following recurrence relation.
		\begin{align*}
			c_{i,j}^{(0)}&= \begin{cases}
				\sum \limits_{s=j}^{n}|V_{0,s}|, & \text{if} \ i=0,\\
				\sum \limits_{r=0}^{i}|V_{r,n}|, & \text{if} \ j=n, \\
				1+|V_{0,j}|+|V_{i,n}|, & otherwise,
			\end{cases}
		\end{align*}
		and for $ \ell \in \{1, 2, \ldots, d_{i,j}\},$\\
		\begin{align*}
			c_{i,j}^{(\ell)}&=\begin{cases}\sum \limits_{r=\ell}^{i}(|V_{r,j}|-\delta_{r,i}) \cdot  \big(c_{r-1,j+1}^{(\ell-1)}-\delta_{\ell-1,0}\big) \\
				\hspace{1cm} +\sum \limits_{s=j+1}^{n-\ell}|V_{i,s}| \cdot \big(c_{i-1,s+1}^{(\ell-1)}-\delta_{\ell-1,0}\big),& \text{when } \ell < n-j,\\
				\sum \limits_{r=\ell}^{i}(|V_{r,j}|-\delta_{r,i}) \cdot  \big(c_{r-1,j+1}^{(\ell-1)}-\delta_{\ell-1,0}\big), & \text{when } \ell = n-j,
			\end{cases}
		\end{align*}
		where $\delta_{i,j}$ is the Kronecker delta function.\\
		Moreover, 
		\begin{align*}
			f_{0}^{\V}(\ig)&=\begin{cases}
				\sum \limits_{s=0}^{n}|V_{0,s}|, & \text{if} \ m=0,\\
				\sum \limits_{r=0}^{m}|V_{r,0}|, & \text{if} \ n=0, \\
				1 +|V_{0,0}|+|V_{m,n}|,& otherwise,
			\end{cases}
		\end{align*}
		and for $t \in \{1,2, \ldots,  d-1\}$,
		\begin{align*}
			f_{t}^{\V}(\ig)&=\sum \limits_{r=t}^{m}(|V_{r,0}|-\delta_{r,m}) \cdot (c_{r-1,1}^{(t-1)}-\delta_{t-1,0}) +\sum \limits_{s=1}^{n-t}|V_{m,s}| \cdot (c_{m-1,s+1}^{(t-1)}-\delta_{t-1,0}),\\  \text{ and, }
			f_{d}^{\V}(\ig)&=\begin{cases}\sum \limits_{r=d}^{m}(|V_{r,0}|-\delta_{r,m}) \cdot(c_{r-1,1}^{(d-1)}-\delta_{d-1,0})\\
				\hspace{1cm}+\sum \limits_{s=1}^{n-t}|V_{m,s}| \cdot(c_{m-1,s+1}^{(d-1)}-\delta_{d-1,0}), & \text{ when } m<n,\\
				\sum \limits_{r=d}^{m}(|V_{r,0}|-\delta_{r,m})\cdot (c_{r-1,1}^{(d-1)}-\delta_{d-1,0}), & \text{ when } m\geq n.
			\end{cases}
		\end{align*}
	\end{proposition}
	
	\begin{proof}
		
		Let $m\geq 1$ and $n\geq 1$. Let us consider the graph $G_{i,j} \in \g_{i,n-j}$, for $i \in \{0,1, \ldots,m-1\}$ and $j \in \{1,2, \ldots, n\}$. 
		Let $v \in V_{i,j}$ and thus $N_{G_{i,j}}(v)=\bigsqcup \limits_{r=0}^{i-1}V_{r,j} \sqcup (V_{i,j} \setminus \{v\}) \bigsqcup \limits_{s=j+1}^{n}V_{i,s}$ (see Figure~\ref{fig g_ij graph}). Let $u \in N_{G_{i,j}}(v)$ be a non-universal vertex.
		
		If $i=0$, then $c_{0,j}^{(0)}=f_{0}^{\V'}(\ii(G_{0,j})).$
		Since $G_{0,j}$ is a complete graph, from Observation~\ref{lmcom}, the number of critical simplices of dimension $0$ in $\ii(G_{0,j})$ is the same as the number of all the vertices of $G_{0,j}$.
		Thus, \[c_{0,j}^{(0)}=\sum \limits_{s=j}^{n}|V_{0,s}|.\]
		
		If $j=n$, then $c_{i,n}^{(0)}=f_{0}^{\V'}(\ii(G_{i,n})).$
		Since $G_{i,n}$ is a complete graph, it follows from a similar argument given above that
		\[c_{i,n}^{(0)}=\sum \limits_{r=0}^{i}|V_{r,n}|.\]
		
		For $i \in \{1,2, \ldots,m-1\}$ and $j \in \{1,2, \ldots, n-1\}$, since $c_{i,j}^{(0)}=f_{0}^{\V'}(\ii(G_{i,j}))$, from Theorem~\ref{mth}, we get
		\begin{align*}
			c_{i,j}^{(0)}&=  1+  (\text{number of universal vertices of } G_{i,j})\\
			&=1+ |V_{0,j}|+|V_{i,n}|.
		\end{align*}
		For $u \in N_{G_{i,j}}(v)$, let $\vu$ be an acyclic matching on $\igij$, constructed recursively in the proof of Theorem~\ref{gridth}. Since $c_{i,j}^{(1)}=f_{1}^{\V'}(\ii(G_{i,j}))$, applying Theorem~\ref{mth}, we get
		\begin{align*}
			c_{i,j}^{(1)}&= \sum \limits_{u \in N_{G_{i,j}}(v)} f_{0}^{\vu}(\igij) - \big(|N_{G_{i,j}}(v)|-(\text{number of universal vertices of } G_{i,j})\big).
		\end{align*}
		Now from Theorem~\ref{gridth}, it follows that 
		\begin{align*}
			c_{i,j}^{(1)}&=\sum \limits_{r=1}^{i} \sum \limits_{u \in V_{r,j}\setminus\{v\}}f_{0}^{\vu}(\igij)+ \sum \limits_{s=j+1}^{n-1} \sum \limits_{u \in V_{i,s}}f_{0}^{\vu}(\igij)\\
			&\hspace{0.4cm}-\big(|N_{G_{i,j}}(v)|-\big(\text{number of universal vertices of } G_{i,j})\big).
		\end{align*}
		Since, for $u \in V_{r,j} \setminus \{v\}$, $G_{i,j}-N_{G_{i,j}}[u]=G_{r-1,j+1}$, and for $u \in V_{i,s}$, $G_{i,j}-N_{G_{i,j}}[u]=G_{i-1,s+1}$, we get
		\begin{align*}
			c_{i,j}^{(1)}&=\sum \limits_{r=1}^{i} \sum \limits_{u \in V_{r,j}\setminus\{v\}}c_{r-1,j+1}^{(0)}+ \sum \limits_{s=j+1}^{n-1} \sum \limits_{u \in V_{i,s}}c_{i-1,s+1}^{(0)}\\
			&\hspace{0.4cm}-\big(|N_{G_{i,j}}(v)|-\big(\text{number of universal vertices of } G_{i,j})\big)\\
			&=\sum \limits_{r=1}^{i}(|V_{r,j}|-\delta_{r,i}) \cdot c_{r-1,j+1}^{(0)}+
			\sum \limits_{s=j+1}^{n-1}|V_{i,s}| \cdot c_{i-1,s+1}^{(0)}\\
			&\hspace{0.4cm}-\big(|N_{G_{i,j}}(v)|-\big(\text{number of universal vertices of } G_{i,j})\big).
		\end{align*}
		\begin{align*}
			\text{Now, }&|N_{G_{i,j}}(v)|-\big(\text{number of universal vertices of } G_{i,j})\\
			=&\sum \limits_{r=0}^{i-1}|V_{r,j}|+|V_{i,j}|-1+ \sum \limits_{s=j+1}^{n}|V_{i,s}|-(|V_{0,j}|+|V_{i,n}|)
			=\sum \limits_{r=1}^{i}(|V_{r,j}|-\delta_{r,i})+ \sum \limits_{s=j+1}^{n-1}|V_{i,s}|.
		\end{align*}
		Therefore,
		\begin{align*}
			c_{i,j}^{(1)}&=\sum \limits_{r=1}^{i}(|V_{r,j}|-\delta_{r,i}) \cdot c_{r-1,j+1}^{(0)}+
			\sum \limits_{s=j+1}^{n-1}|V_{i,s}| \cdot c_{i-1,s+1}^{(0)}\\
			&\hspace{0.4cm}-\big(\sum \limits_{r=1}^{i}(|V_{r,j}|-\delta_{r,i})+ \sum \limits_{s=j+1}^{n-1}|V_{i,s}|\big)\\
			&=\sum \limits_{r=1}^{i}(|V_{r,j}|-\delta_{r,i}) \cdot \big(c_{r-1,j+1}^{(0)}-1\big)+
			\sum \limits_{s=j+1}^{n-1}|V_{i,s}| \cdot \big(c_{i-1,s+1}^{(0)}-1\big).
		\end{align*}
		Similarly, for $\ell \in \{2,3, \ldots, d_{i,j}\}$, we may verify that
		
		\begin{align*}
			c_{i,j}^{(\ell)} &=\begin{cases}
				\sum \limits_{r=\ell}^{i}(|V_{r,j}|-\delta_{r,i}) \cdot c_{r-1,j+1}^{(\ell-1)}+\sum \limits_{s=j+1}^{n-\ell}|V_{i,s}| \cdot c_{i-1,s+1}^{(\ell-1)} , & \text{if } \ell <n-j,\\
				\sum \limits_{r=\ell}^{i}(|V_{r,j}|-\delta_{r,i}) \cdot c_{r-1,j+1}^{(\ell-1)}, & \text{if } \ell=n-j.
			\end{cases}
		\end{align*}
		
		Finally, an analogous computation, with $i=m$ and $j=0$, yields the recursion formula for the critical $f$-vector of $\ig$ (with respect to $\V$).
	\end{proof}

	\bibliographystyle{plainurl}
	\bibliography{references}

\end{document}